\documentclass[11pt]{amsart}
\usepackage{geometry}                
\usepackage{amssymb}

\usepackage[dvips]{graphics}
\usepackage[dvips]{graphicx}

\usepackage{pstricks}
\usepackage{amsmath}
\usepackage{amsfonts}
\usepackage{latexsym}
\usepackage{amssymb}
\usepackage[cmtip,arrow, all]{xy}
\usepackage{enumerate}

\newtheorem{theorem}{Theorem}
\newtheorem{proposition}{Proposition}[section]
\newtheorem{lemma}[proposition]{Lemma}
\newtheorem{corollary}[proposition]{Corollary}
\newtheorem{definition}[proposition]{Definition}
\newtheorem{remark}[proposition]{Remark}

\newcommand {\Ext}{{\rm {Ext}}}
\newcommand {\Hom}{{\rm {Hom}}}
\newcommand{\C}{{{\mathcal C}}}
\newcommand{\R}{{\mathbf R}}
\newcommand{\Z}{{{\mathbf Z}}}

\newcommand{\cd}{{\rm cd}}

\newcommand{\im}{{\rm {Im}}}
\newcommand{\cat}{\sf {cat}}

\newcommand{{\pdone}}{{\rm Proj \, \, Dim_1}}

\newcommand{\tc}{{\sf {TC}}}
\newcommand{\hd}{{\sf {gd}}}
\newcommand{\wgt}{{\rm {wgt}}}
\newcommand{\vv}{\mathfrak {v}}
\newcommand{\vp}{\mathfrak {b}}
\newcommand{\zpp}{{{\Z[\pi\times\pi]}}}
\newcommand{\zp}{{{\Z[\pi]}}}

\newcommand{\ev}{{\rm \mathbf {ev}}}
\newcommand{\cps}{{\mathcal C}_{\pi^s}'}
\newcommand{\zc}{\Z[C]}

\graphicspath{%
    {converted_graphics/}
    {/}
}

\begin{document}

\title{On the topological complexity of aspherical spaces}         
\author{Michael Farber and Stephan Mescher}        
\date{\today}          
\maketitle
\begin{abstract} 
The well-known theorem of Eilenberg and Ganea \cite{EG} expresses the Lusternik-Schnirelmann category of an aspherical space $K(\pi, 1)$ as the cohomological dimension of the group $\pi$. In this paper we study a similar problem of determining algebraically the topological complexity of the Eilenberg-MacLane spaces $K(\pi, 1)$. One of our main results states that in the case when the group $\pi$ is hyperbolic in the sense of Gromov the topological complexity 
$\tc(K(\pi, 1))$ either equals or is by one larger than the cohomological dimension of $\pi\times \pi$.
We approach the problem by studying {\it essential cohomology classes}, i.e. classes which can be obtained from the powers of the canonical class 
(defined in \cite{CF}) via  coefficient homomorphisms. 
We describe a spectral sequence which allows to specify a full set of obstructions for a cohomology class to be essential. 
In the case of a hyperbolic group we establish a vanishing property of this spectral sequence which leads to the main result. 
\vskip 0.3 cm
\noindent MSC: 55M99

\end{abstract}
%
%
%
%
%
\section{Introduction}

In this paper we study a numerical topological invariant $\tc (X)$ of a topological space $X$, originally introduced in \cite{F}, see also \cite{FarberTCsurvey}, \cite{Finv}. 
The concept of $\tc(X)$ is related to the motion planning problem of robotics where a system (robot) has to be programmed to be able to move autonomously from any initial state to any final state. In this situation a motion of the system is represented by a continuous path 
in the configuration space $X$ and a motion planning algorithm is a section of the path fibration 
\begin{eqnarray}\label{fibration}
p: PX\to X\times X, \quad p(\gamma) = (\gamma(0), \gamma(1)).
\end{eqnarray}
Here $PX$ denotes the space of all continuous paths $\gamma: [0, 1]\to X$ equipped with the compact-open topology. The topological complexity $\tc(X)$ is an integer reflecting the complexity of 
this fibration, it has several different characterisations, see \cite{FarberTCsurvey}. Intuitively, $\tc(X)$ is a measure of the navigational complexity of $X$ viewed as the configuration space of a system. 
$\tc(X)$ is similar in spirit to the classical Lusternik - Schnirelmann category $\cat(X)$.
The invariants $\tc(X)$ and $\cat(X)$ are special cases of a more general notion of the genus of a fibration introduced by A. Schwarz \cite{Sch}. A recent survey of 
the concept $\tc(X)$ and robot motion planning algorithms in interesting configuration spaces can be found in \cite{Frecent}.

\begin{definition}\label{def1}{\rm
Given a path-connected topological space $X$, the topological complexity of $X$ is defined as the minimal number 
$\tc(X)=k$ such that the Cartesian product $X \times X$ can be covered by $k$ open subsets
$X \times X = U_1 \cup U_2 \cup \dots U_k$ with the property that for any $i = 1,2,...,k$ there exists a continuous section
$s_i: U_i \to PX,$
$\pi\circ s_i={\rm id}$, over $U_i$. If no such $k$ exists we will set $\tc(X)=\infty$.
}
\end{definition}

Note that in the mathematical literature there is also a {\it reduced} version of the topological complexity which is one less compared to the one we are dealing with in this paper. 
\vskip 0.3cm

One of the main properties of $\tc(X)$ is its {\it homotopy invariance} \cite{F}, i.e. $\tc(X)$ depends only on the homotopy type of $X$. 
This property is helpful for the task of computing $\tc(X)$ in various examples since cohomological tools can be employed. 
In the case when the configuration space $X$ is aspherical, i.e. $\pi_i(X)=0$ for all $i>1$, the number $\tc(X)$ depends only on the fundamental group 
$\pi=\pi_1(X)$ and it was observed in \cite{FarberTCsurvey} that one has to be able to express $\tc(X)$ in terms of algebraic properties of the group
$\pi$ alone. 

A similar question for the Lusternik - Schnirelmann category $\cat(X)$ was solved by S. Eilenberg and T. Ganea in 1957 in the seminal paper \cite{EG}. 
Their theorem relates $\cat(X)$ and the cohomological dimension of the fundamental group $\pi$ of $X$. 

The problem of computing $\tc(K(\pi,1))$ as an algebraic invariant of the group $\pi$ attracted attention of many mathematicians. Although no general answer is presently known, many interesting results were obtained. 

The initial papers \cite{F}, \cite{FarberTCsurvey} contained computations of $\tc(X)$ for graphs, closed orientable surfaces and tori. In \cite{FarberYuzvinsky} the number $\tc(X)$ was computed for the case when $X$ is the configuration space of many particles moving on the plane without collisions. 
D. Cohen and G. Pruidze  \cite{CohenPruidze} calculated the topological complexity of complements of general position arrangements and Eilenberg -- MacLane spaces associated to certain right-angled Artin groups.

As a recent breakthrough, the topological complexity of closed non-orientable surfaces of genus $g \geq 2$ has only recently been computed by A. Dranishnikov for $g \geq 4$ in \cite{Dranish} and by D. Cohen and L. Vandembroucq for $g=2, 3$ in \cite{CohenVandem}. In both these articles it is shown that $\tc(K(\pi,1))$ attains its maximum, i.e. coincides with $\mathrm{cd}(\pi \times \pi)+1$.

The estimates of M. Grant \cite{Grant} give good upper bounds for $\tc(K(\pi, 1))$ for nilpotent fundamental groups $\pi$. 
In  \cite{GrantLuptonOprea}, M. Grant, G. Lupton and J. Oprea proved that $\tc(K(\pi,1))$ is bounded below by the cohomological dimension 
of $A \times B$ where $A$ and $B$ are  subgroups of $\pi$ whose conjugates intersect trivially. Using these estimates, M. Grant and D. Recio-Mitter \cite{GrantRecio}
have computed $\tc(K(\pi,1))$ for certain subgroups of Artin's braid groups. 

Yuli Rudyak \cite{Rudyak} showed that for any pair of positive integers $k, \ell$ satisfying $k\le \ell \le 2k$ there exists a finitely presented group $\pi$ such that $\cd(\pi)=k$ and $\tc(K(\pi, 1))= \ell+1.$

\vskip 0.5 cm

The authors thank Boris Goldfarb, Sasha Dranishnikov and John Mackay for generously sharing their expertise. We also thank Mark Grant for useful discussions. 


%
%
%
%
%
%
%

\section{Statements of the main results}

In this section we state the main results obtained in the present paper. 

\begin{theorem}\label{thm1} Let $X$ be a connected aspherical finite cell complex with hyperbolic fundamental group $\pi=\pi_1(X)$.
Then the topological complexity $\tc(X)$ equals either $\cd(\pi\times\pi)$ or $\cd(\pi\times\pi)+1$.
\end{theorem}


The symbol $\cd(\pi\times\pi)$ stands for the cohomological dimension of $\pi\times\pi$ and similarly for $\cd(\pi)$. 
The Eilenberg-Ganea theorem \cite{Brown} states that
$\cd(\pi\times\pi) =\hd(\pi\times\pi)$ (with one possible exception) and hence the general dimensional upper bound for $\tc(X)$ (see \cite{F}, \cite{Finv}) gives 
\begin{equation}\label{general}\tc(X)\le \hd(\pi\times\pi)+1=\cd(\pi\times\pi)+1.\end{equation}  
Here $\hd(\pi\times\pi)$ denotes {\it the geometric dimension} of $\pi\times\pi$, i.e. the minimal dimension of a cell complex with fundamental group $\pi\times\pi$. Thus Theorem \ref{thm1} essentially states that 
the topological complexity 
$\tc(K(\pi, 1))$, where $\pi$ is hyperbolic, is either maximal (as allowed by the dimensional upper bound) or is by one smaller than the maximum. 

The notion of a hyperbolic group was introduced by M. Gromov in \cite{Ghyp}; we also refer the reader to the monograph \cite{BH}. 
Hyperbolic groups are \lq\lq typical\rq\rq, i.e. they appear with probability tending to 1, in many models of random groups including Gromov's 
well-known  density model \cite{Gromov}, \cite{Ollivier}. 

As an example, consider the case of the fundamental group $\pi=\pi_1(\Sigma_g)$ of a closed orientable surface of 
genus $g\ge 2$. It is torsion-free hyperbolic and $\tc(\Sigma_g)= \cd(\pi\times\pi)+1=5$, see \cite{Finv}, in accordance with the maximal option of Theorem \ref{thm1}. Similarly, if $\pi=F_\mu$ is a free group on $\mu$ generators then, according to \cite{Finv}, Proposition 4.42, $\tc(K(F_\mu,1))=3$ (for $\mu>1$) and $\cd(F_\mu\times F_\mu)=2$; here again Theorem \ref{thm1} is satisfied in the maximal version. 

The only known to us example of an aspherical space $X=K(\pi, 1)$ with $\pi$ hyperbolic where 
$\tc(X) = \cd(\pi\times\pi)$ is the case of the circle $X=S^1$. It would be interesting to learn if some other examples of this type exist. 

Theorem \ref{thm1} follows from the following statement:

\begin{theorem}\label{thm11} Let $X$ be a connected aspherical finite cell complex with fundamental group $\pi=\pi_1(X)$. 
Suppose that (1) the centraliser of any nontrivial element $g\in \pi$ is cyclic and (2) $\cd(\pi\times\pi)>\cd(\pi)$. 
Then the topological complexity $\tc(X)$ equals either $\cd(\pi\times\pi)$ or $\cd(\pi\times\pi)+1$.
\end{theorem}

We do not know examples of finitely presented groups $\pi$ such that $\cd(\pi\times \pi) =\cd(\pi)$, i.e. such that the assumption (2) of Theorem \ref{thm11} is violated. 
A. Dranishnikov \cite{Dra} constructed examples with $\cd(\pi_1\times\pi_2) <\cd(\pi_1)+\cd(\pi_2)$; see also \cite{MWD}, page 157. 
In  \cite{Dra} he also proved that $\cd(\pi\times\pi)= 2\cd(\pi)$ for any Coxeter group $\pi$. 

To state another main result of this paper we need to recall the notion of $\tc$-{\it weight} of cohomology classes as introduced in \cite{FG}; this notion is similar but not identical to the concept of $\tc$-weight introduced in \cite{Finv}, \S 4.5; both these notions were inspired by the notion of 
category weight of cohomology classes initiated by E. Fadell and S. Husseini \cite{FH}. 

\begin{definition}\label{defweight}
Let $\alpha \in H^\ast(X\times X, A)$ be a cohomology class, where $A$ is a local coefficient system on $X\times X$. We say that $\alpha$ has weight $k\ge 0$
(notation $\wgt (\alpha)=k$) 
if $k$ is the largest integer with the property that for any continuous map $f: Y\to X\times X$ (where $Y$ is a topological space) one has $f^\ast(\alpha)=0\in H^\ast(Y, f^\ast(A))$ provided the space $Y$ admits an open cover $U_1\cup U_2\cup\dots\cup U_k=Y$ such that each restriction map $f| U_j\, : \, U_j\to X\times X$ admits a continuous lift $U_j\to PX$ into the path-space fibration (\ref{fibration}).
\end{definition}


A cohomology class $\alpha\in H^\ast(X\times X, A)$ has a positive weight $\wgt(\alpha)\ge 1$ if and only if $\alpha$ is a {\it zero-divisor}, i.e. if its restriction to the 
diagonal $\Delta_X\subset X\times X$  vanishes,
$$0=\alpha\, |\, {\Delta_X}\in H^\ast(X, \tilde A),$$ 
see \cite{FG}, page 3341. Here $\tilde A$ denotes the restriction local system $A\, |\,  {\Delta_X}$. Note that in \cite{FG} the authors considered untwisted coefficients but all the arguments automatically extend to general local coefficient systems. In particular, by Proposition 2 from \cite{FG} we have
\begin{eqnarray}\label{add}\wgt(\alpha_1\cup \alpha_2) \ge \wgt(\alpha_1) + \wgt(\alpha_2)\end{eqnarray}
for cohomology classes $\alpha_i\in H^{d_i}(X\times X, A_i)$, $i=1, 2$, where the cup-product $\alpha_1\cup \alpha_2$ lies in $H^{d_1+d_2}(X\times X, A_1\otimes_\Z A_2)$.

\begin{theorem}\label{thm2} Let $X$ be a connected aspherical finite cell complex. Suppose that the fundamental group $\pi=\pi_1(X)$ is such that 
the centraliser of any nontrivial element $g\in \pi$ is infinite cyclic. Then any degree $n$ zero-divisor $\alpha\in H^n(X\times X, A)$, where $n\ge 1$, has weight 
$\wgt(\alpha)\ge n-1$. 
\end{theorem}

For obvious reasons this theorem is automatically true for $n=1, 2$; it becomes meaningful only for $n>2$. 


Here is a useful corollary of Theorem \ref{thm2}:

\begin{theorem}\label{cor4}
Under the assumptions of Theorem \ref{thm11}, one has $$\vv^{n-1}\not=0\in H^{n-1}(\pi\times\pi, I^{n-1})$$ where $n=\cd(\pi\times\pi)$. 
Here $\vv\in H^1(\pi\times\pi, I)$ denotes the canonical class, see \S \ref{secan} below. 
\end{theorem}

The statement of Theorem \ref{cor4} becomes false if we remove the assumptions on the fundamental group. 
For example in the case of an abelian group $\pi=\Z^k$ (see \S \ref{abel}) we have $n =\cd(\pi\times\pi) = 2k$ and $\vv^k$ is the highest nontrivial 
power of the canonical class.  
\vskip 0.3cm

{\bf Question}: Let $\pi$ be a noncommutative hyperbolic group and let $n$ denote $\cd(\pi\times\pi)$. 
Is it true that the $n$-th power of the canonical class $\vv^n\in H^n(\pi\times\pi, I^n)$ is nonzero, $\vv^n\not=0$? 
\vskip 0.3cm

A positive answer to this question would imply that for any noncommutative hyperbolic group $\pi$ one has $\tc(K(\pi, 1))=\cd(\pi\times\pi) +1$. 
\vskip 0.3cm
The proofs of Theorems \ref{thm1}, \ref{thm11}, \ref{thm2} and \ref{cor4} are given in \S  \ref{proofs}. 

In \S \ref{sym}, we present an application of Theorem \ref{thm2} to the topological complexity of symplectically aspherical manifolds.

\section{The canonical class}\label{secan}

First we fix notations which will be used in this paper.
We shall consider a discrete torsion-free group $\pi$ with unit element $e\in \pi$ and left modules $M$ over the group ring $\Z[\pi\times\pi]$. 
Any such module $M$ can be equivalently viewed as a $\pi-\pi$-bimodule using the convention $$(g, h)\cdot m= gmh^{-1}$$ for $g, h\in \pi$ and $m\in M$.  
Recall that for two left 
$\Z[\pi\times\pi]$-modules $A$ and $B$ the module $\Hom_{\Z}(A, B)$ has a canonical $\Z[\pi\times\pi]$-module structure given by $((g, h)\cdot f)(a) = gf(g^{-1}ah)h^{-1}$ where
$g, h\in \pi$, $a\in A$ and $f:A\to B$ is a group homomorphism. 

Besides, the tensor product $A\otimes_\Z B$ has a left $\Z[\pi\times\pi]$-module structure given by 
$(g, h)\cdot (a\otimes b) = (gah^{-1})\otimes (gbh^{-1})$ where $g, h\in \pi$ and $a\in A$, $b\in B$;
we shall refer to this action as {\it the diagonal action}. 

For a left 
$\Z[\pi\times\pi]$-module $A$ we shall denote by $\tilde A$ the same abelian group viewed as a $\Z[\pi]$-module via {\it the conjugation action}, i.e. 
$g\cdot a= gag^{-1}$ for $g\in \pi$ and $a\in A$.

The group ring $\Z[\pi]$ is a $\Z[\pi\times\pi]$-module with respect to the action 
$$(g, h)\cdot a= gah^{-1}, \quad \mbox{where}\quad g, h, a\in \pi.$$
The augmentation homomorphism 
$\epsilon: \Z[\pi]\to \Z$
is a $\Z[\pi\times\pi]$-homomorphism where we consider the trivial $\Z[\pi\times\pi]$-module structure on $\Z$. The {\it augmentation ideal} $I=\ker \epsilon$ is hence a $\Z[\pi\times\pi]$-module and we have a short exact sequence of $\Z[\pi\times\pi]$-modules
\begin{equation}\label{can}
0\to I\to \Z[\pi]\stackrel\epsilon\to \Z\to 0.\end{equation}

In this paper we shall use the the formalism (described in \cite{HS}, Chapter IV, \S 9)
which associates a well defined class 
$$\theta\in \Ext^n_R(M, N)$$  with any exact sequence 
$$0\to N\to L_n\to L_{n-1}\to \dots\to L_1\to M\to 0$$ of left $R$-modules and $R$-homomorphisms, where $R$ is a ring. This construction can be briefly summarised as follows.
If $$ \, \, \, \dots\to C_2\to C_1\to C_0\to M\to 0$$ is a projective resolution of $M$ over $R$, one obtains a commutative diagram
\begin{equation*}
\begin{array}{crcccccccccccc}
C_{n+1} & \to & C_n& \to & C_{n-1} & \to &\dots & C_1 & \to & C_0 & \to & M& \to &0\\ 
\downarrow && {} \downarrow f && \downarrow && \dots & \downarrow && \downarrow && {} \downarrow = &&\\ 
0& \to & N &\to & L_n &\to  &\dots & L_2 &\to  & L_1& \to & M & \to & 0.
\end{array}
\end{equation*}
The homomorphism $f: C_n \to N$ is a cocycle of the complex $\Hom_R(C_\ast, N)$, which is defined uniquely up to chain homotopy. 
The class $\theta$ is the cohomology class of this cocycle
$$\theta\, =\,  \{f\}\in H^n(\Hom_R(C_\ast, N))=\Ext^n_R(M, N).$$
Note that the definition of Bourbaki (see \cite{BB} \S 7, n. 3) is slightly different since \cite{BB} introduces additionally a sign factor $(-1)^{n(n+1)/2}$.

An important role plays the class 
$${\vv}\in \Ext_{\Z[\pi\times\pi]}^1(\Z, I)=H^1(\pi\times\pi, I)$$
associated with the exact sequence (\ref{can}). 
It was introduced in \cite{CF} under the name of {\it the canonical class}. 

To describe the cocycle representing the canonical class $\vv$ consider the bar resolution $C_\ast$ of $\Z$ over $\zpp$, see \cite{Brown}, page 19. Here
$\dots \to C_1\stackrel{d}\to C_0\to \Z\to 0$ where $C_0$ is a free $\zpp$-module generated by the symbol $[\, \,]$ and $C_1$ is the free $\zpp$-module generated 
by the symbols $[(g, h)]$ for all $(g, h)\in \pi\times \pi$. The boundary operator $d$ acts by 
$$d[(g, h)]= \left((g,h)-1\right)[\,\,].$$
We obtain the chain map 
$$
\begin{array}{ccccccccc}
C_2 & \to & C_1& \stackrel d \to& C_0& \to & \Z&\to &  0 \\
\downarrow &&{\, } \downarrow f& & {\,} \downarrow \mu& & {\, }\downarrow =&&\\
0&  \to & I&\to &\zp &\stackrel{\epsilon}\to & \Z&\to &0
\end{array}
$$
where $\mu([\, \, ])=1$ and 
\begin{equation}\label{cocyclev}
f([(g, h)]) = gh^{-1}-1 \in I.
\end{equation}
Thus, the cocycle $f:C_1\to I$ is given by the crossed homomorphism (\ref{cocyclev}).
Comparing with \cite{CF}, we see that the definition of the canonical class given above coincides with the definition 
given in \cite{CF}, page 110. 

We shall also describe the cocycle representing the canonical class in {\it the homogeneous standard resolution} of $\pi\times \pi$, see \cite{Brown}, page 18:
\begin{equation}\label{homstan}
\dots \to C'_2\stackrel{d} \to C'_1\stackrel{d}\to C'_0\to \Z\to 0.
\end{equation} Here $C'_i$ is a free $\Z$-module generated by the $(i+1)$-tuples $((g_0, h_0), \dots, (g_i, h_i))$ with
$g_j, h_j\in \pi$ for any $j=0, 1, \dots, i$. Using (\ref{cocyclev}) we obtain that the cocycle $f': C'_1\to I$ representing the canonical class $\vv$ is given by the formula
\begin{equation}\label{cocyclevv}
f'((g_0, h_0), (g_1, h_1)) = g_1h_1^{-1} - g_0h_0^{-1}, \quad g_j, h_j\in \pi.
\end{equation}

The canonical class $\vv$ is closely related to {\it the Berstein-Schwarz class} (see \cite{Ber}, \cite{DR}); the latter is crucial for the study of the Lusternik-Schnirelmann category ${\sf {cat}}$. The Berstein-Schwarz class can be defined as the class $$\vp\in \Ext^1_{\Z[\pi]}(\Z, I)=H^1(\pi, I)$$ which corresponds to the exact sequence (\ref{can}) viewed as a sequence of left $\Z[\pi]$-modules via the left action of $\pi=\pi\times 1\subset \pi\times \pi$. For future reference we state 
\begin{equation}\label{v2b}
\vv \, |\, \pi\times 1 \, =\, \vp;
\end{equation}
here $\pi\times 1\subset \pi\times \pi$ denotes the left factor viewed as a subgroup. 

The main properties of the canonical class $\vv$ are as follows. 

Let $X$ be a finite
connected cell complex with fundamental group $\pi_1(X)=\pi$. 
We may view $I$ as a local coefficient system over $X\times X$ and form the 
cup-product $\vv\cup\vv\cup\dots\cup \vv=\vv^{k}$ ($k$ times) which lies in the cohomology group 
$$\vv^{k} \, \in \, H^k(X\times X, I^{k}),$$ where $I^{k}$ denotes the tensor product $I\otimes_\Z I\otimes_\Z \dots \otimes_\Z I$ of $k$ copies of $I$ viewed as a left $\Z[\pi\times\pi]$-module via the diagonal action as explained above. 
Let $n$ denote the dimension of $X$. 
It is known that in general the topological complexity satisfies $\tc(X)\le 2n +1$ and the equality 
$$\tc(X)\, =\, 2\dim(X) +1=2n+1$$
happens if and only if $\vv^{2n}\not=0$; see \cite{CF}, Theorem 7. 

Another important property of $\vv$ is that it is {\it a zero-divisor}, i.e. 
\begin{equation}\label{can1}
\vv\, |\, \Delta_\pi=0\, \in \, H^1(\pi, \tilde I)
\end{equation}
where $\Delta_\pi \subset \pi\times \pi$ is the diagonal subgroup, $\Delta_\pi=\{(g, g);g\in \pi\}$. This immediately follows from the observation that 
the cocycle $f$ representing $\vv$ (see (\ref{cocyclev})) vanishes on the diagonal $\Delta_\pi$. 

Our next goal is to describe an exact sequence representing the power $\vv^n$ of the canonical class. 
The exact sequence (\ref{can}) splits over $\Z$ and hence for any left $\Z[\pi\times\pi]$-module $M$ tensoring over $\Z$ we obtain an exact sequence
\begin{equation}\label{timesm}
0\to I\otimes_\Z M\to \Z[\pi]\otimes_\Z M\stackrel\epsilon\to M\to 0.
\end{equation}
In (\ref{timesm}) we consider the diagonal action of $\pi\times\pi$ on the tensor products.
Taking here $M=I^s$ where $I^s=I\otimes_\Z I \otimes_\Z  I \dots \otimes_\Z I$ we obtain a short exact sequence
\begin{equation}\label{s}
0\to I^{s+1}\stackrel{i\otimes 1}\to \zp\otimes_\Z I^s\stackrel {\epsilon\otimes 1}\to I^s\to 0.
\end{equation}
Here $i:I\to \zp$ is the inclusion and $\epsilon: \zp\to \Z$ is the augmentation. 
Splicing exact sequences (\ref{s}) for $s=0, 1, \dots, n-1$ we obtain an exact sequence
\begin{equation}\label{vn}
0\to I^n \to \zp\otimes_\Z I^{n-1} \to \zp\otimes_\Z I^{n-2} \to\dots \zp\otimes_\Z I\to \zp\to \Z \to 0.
\end{equation}
\begin{lemma}\label{vvrepresented}
The cohomology class $$\vv^n\in H^n(\pi\times \pi, I^n)= \Ext^n_\zpp (\Z, I^n)$$
is represented by the exact sequence (\ref{vn}). 
\end{lemma}
\begin{proof} Consider again the homogeneous standard resolution (\ref{homstan}) 
of $\pi\times \pi$. Define $\zpp$-homomorphisms
\begin{equation}\label{kapa}
\kappa_j: C'_j \to \zp\otimes_\Z I^j, \quad \mbox{where}\quad j=0, 1, \dots, n-1,
\end{equation}
by the formula
\begin{equation}\label{kj}
\kappa_j((g_0, h_0), (g_1, h_1), \dots, (g_j, h_j)) = x_0\otimes (x_1-x_0)\otimes \dots \otimes (x_j-x_{j-1}),
\end{equation}
where the symbol $x_i$ denotes $g_ih_i^{-1}\in \pi$ for $i=0, \dots, j$. 

We claim that the homomorphisms $\kappa_j$, for $j=0, \dots, n$, determine a chain map from the homogeneous standard resolution (\ref{homstan}) into the exact sequence (\ref{vn}). In other words, we want to show that 
\begin{equation}\label{kjj}
\kappa_{j-1} (d((g_0, h_0), (g_1, h_1), \dots, (g_j, h_j)))= (x_1-x_0)\otimes \dots \otimes (x_j-x_{j-1}).
\end{equation}
This statement is obvious for $j=1$. To prove it for $j>1$ we apply induction on $j$. 
Denoting $$\Pi_j(x_0, x_1, \dots, x_{j-1}) = x_0\otimes (x_1-x_0)\otimes \dots \otimes (x_{j-1}-x_{j-2})$$ we may write 
$$\kappa_{j-1} (d((g_0, h_0), (g_1, h_1), \dots, (g_j, h_j))) = \sum_{i=0}^j (-1)^i \Pi_j(x_0, \dots, \hat x_i, \dots , x_j).$$
The last two terms in this sum (for $i=j-1$ and $i=j$) sum up to 
$$(-1)^{j-1} x_0\otimes (x_1-x_0) \otimes \dots (x_{j-2}-x_{j-3})\otimes (x_j-x_{j-1}).$$
Thus we see that the LHS of (\ref{kjj}) can be written as 
$$\left[\sum_{i=0}^{j-1} (-1)^i \Pi_{j-1}(x_0, \dots, \hat x_i, \dots, x_{j-1})\right] \otimes (x_j-x_{j-1})$$
and our statement follows by induction. 

The homomorphism $f_n: C_n'\to I^n$ which appears in the commutative diagram
$$
\begin{array}{ccccccccc}
C'_{n+1} & \to & C'_n & \stackrel d \to & C'_{n-1} & \stackrel d \to & C'_{n-2} & \to & \dots\\
&& {\, \, }\downarrow f_n&&{\, \, }\downarrow \kappa_{n-1} && {\, \, } \downarrow \kappa_{n-2} && \\
0 & \to & I^n &  \to & \zp \otimes I^{n-1}  & \to & \zp \otimes I^{n-2}  & \to & \dots
\end{array}
$$
is given by the formula
\begin{equation}
f_n((g_0, h_0), (g_1, h_1), \dots, (g_n, h_n))=  
(x_1-x_0) \otimes (x_2-x_1)\otimes \dots (x_n-x_{n-1})
\end{equation}
where $x_i=g_ih_i^{-1}$. Since the cocycle representing $\vv$ is given by $x_1-x_0$ (see (\ref{cocyclevv})), using the diagonal approximation in the standard complex (see \cite{Brown}, page 108) we find that $f_n$ represents $\vv^n$. \end{proof}

\begin{remark} {\rm Lemma \ref{vvrepresented} also follows by applying Theorems 4.2 and 9.2 from \cite{Maclane}, Chapter VIII.
}
\end{remark} 

The canonical class $\vv$ allows to describe the connecting homomorphisms in cohomology as we shall exploit several times in this paper. 
Let $M$ be a left $\zpp$-module. 
The Bockstein homomorphism 
$$\beta: H^i(\pi\times\pi, M)\to H^{i+1}(\pi\times\pi, I\otimes M)$$ of the exact sequence (\ref{timesm})
acts as follows
\begin{equation}\label{cupm}
\beta(u)= \vv \cup u, \quad \mbox{for}\quad u\in H^i(\pi\times\pi, M).
\end{equation}
This follows from Lemma 5 from \cite{CF} and from \cite{Brown}, chapter V, (3.3).

\section{Universality of the Berstein - Schwarz class}

In the theory of Lusternik - Schnirelmann category an important role plays the following result which was originally stated (without proof) by 
A.S. Schwarz \cite{Sch}, Proposition 34. A recent proof can be found in \cite{DR}. 

\begin{theorem} \label{thm3} For any left $\Z[\pi]$-module $A$ and for any cohomology class $\alpha\in H^n(\pi, A)$  one may find a $\Z[\pi]$-homomorphism 
$\mu: I^n \to A$ such that $\alpha \, = \, \mu_\ast(\vp^n)$. \end{theorem}

Recall that we view the tensor power $I^n = I\otimes_\Z I \otimes_\Z \dots\otimes_\Z I$ as a left $\Z[\pi]$-module using the diagonal action of $\pi$ from the left, i.e. 
$g\cdot (\alpha_1\otimes\alpha_2\otimes\dots\otimes\alpha_n) = g\alpha_1\otimes g\alpha_2\otimes\dots\otimes g\alpha_n$ where $g\in \pi$ and $\alpha_i\in I$ for $i=1, \cdots, n$.

In other words, Theorem \ref{thm3} states that the powers of the Berstein - Schwarz class $\vp^n$ are universal in the sense that any other degree $n$ cohomology class can be obtained from $\vp^n$ by a coefficient homomorphism. This result implies that the Lusternik - Schnirelmann category of an aspherical space is at least $\cd(\pi)+1$. 

We include below a short proof of Theorem \ref{thm3} (following essentially \cite{DR}) for completeness. 
\begin{proof}
First one observes that $I^s$ is a free abelian group and hence $\Z[\pi]\otimes_\Z I^s$ is free as a left $\Z[\pi]$-module; here we apply Corollary 5.7 from chapter III of \cite{Brown}. Hence the exact sequence
\begin{equation}\label{resolution}
\dots \to \Z[\pi]\otimes_\Z I^{n}\to  \Z[\pi]\otimes_\Z I^{n-1} \to \dots \to \Z[\pi]\otimes_\Z I\to  \Z[\pi]\to \Z\to 0
\end{equation}
 is a free resolution of $\Z$ over $\Z[\pi]$. 
The differential of this complex is given by
$$\Z[\pi]\otimes_\Z I^n = \Z[\pi]\otimes_\Z I\otimes_\Z I^{n-1} \stackrel{\epsilon \otimes i \otimes 1}\longrightarrow \Z\otimes_\Z \Z[\pi]\otimes_\Z I^{n-1} = \Z[\pi]\otimes_\Z I^{n-1};$$
here $\epsilon$ is the augmentation and $i: I\to \Z[\pi]$ is the inclusion. 

Using resolution (\ref{resolution}), any degree $n$ cohomology class $\alpha\in H^n(\pi, A)$ can be represented by an $n$-cocycle 
$f: \Z[\pi]\otimes_\Z I^n\to A$ which is a $\Z[\pi]$-homomorphism vanishing on the image $I^{n+1}$ of the boundary homomorphism $\Z[\pi]\otimes_\Z I^{n+1}\to \Z[\pi]\otimes_\Z I^n$. In view of the short exact sequence
$$0\to I^{n+1}\to \Z[\pi]\otimes_\Z I^n \stackrel{\epsilon\otimes 1}\to I^n\to 0$$
we see that {\it there is a 1-1 correspondence between cocycles $f: \Z[\pi]\otimes_\Z I^n\to A$ and homomorphisms $\mu: I^n\to A$. }

Let $\mu: I^n \to A$ be the $\Z[\pi]$-homomorphism corresponding to a cocycle representing the class $\alpha$. 

Using the definition of a class associated to an exact sequence (see beginning of \S \ref{secan}), Lemma \ref{vvrepresented} and formula (\ref{v2b}) we see that the identity map $I^n\to I^n$ corresponds to the $n$-th power of the Berstein - Schwarz class $\vp^n$. Combining all these mentioned results we obtain $\mu_\ast(\vp^n)=\alpha$. 
\end{proof}

\section{Essential cohomology classes}

It is easy to see that the analogue of Theorem \ref{thm3} fails when we consider cohomology classes $\alpha\in H^n(\pi\times \pi, A)$ and ask 
whether such classes can be  obtained from powers of the canonical class $\vv\in H^1(\pi\times \pi, I)$ by coefficient homomorphisms. 
The arguments of the proof of Theorem \ref{thm3} are not applicable since the $\Z[\pi\times\pi]$-modules $\Z[\pi]\otimes_\Z I^s$ are neither free nor projective over the ring
$\Z[\pi\times\pi]$. 

\begin{definition} We shall say that a cohomology class $\alpha\in H^n(\pi\times \pi, A)$ is {\it essential} if there exists a homomorphism of $\Z[\pi\times \pi]$-modules $\mu: I^n \to A$ such that $\mu_\ast(\vv^n)=\alpha$. 
\end{definition}

One wants to have verifiable criteria which guarantee that a given cohomology class $\alpha\in H^n(\pi\times \pi, A)$ is essential. 
Since $\vv$ and all its powers are zero-divisors, it is obvious that any essential class must also be a zero-divisor, i.e. satisfy
$$\alpha\, |\, \Delta_\pi\, =\, 0\in H^n(\pi, \tilde A),$$
see above. For degree one cohomology classes this condition is sufficient, see Lemma \ref{degreeone1}. However, as we shall see,  a degree $n\ge 2$ zero-divisor 
does not need to be essential. 

Clearly, the set of all essential classes in $H^n(\pi\times \pi, A)$ forms a subgroup. 

Moreover, the cup-product of two essential classes $\alpha_i\in H^{n_i}(\pi\times\pi, A_i)$, where $i=1, 2$, is an essential class 
$$\alpha_1\cup \alpha_2\in H^{n_1+n_2}(\pi\times \pi, A_1\otimes_{\Z}A_2).$$
Indeed, suppose $\mu_i: I^{n_i}\to A_i$ are $\Z[\pi\times\pi]$-homomorphisms such that ${\mu_i}_\ast(\vv^{n_i})=\alpha_i$, where $i=1, 2$. 
Then $\mu=\mu_1\otimes\mu_2: I^{n_1}\otimes_\Z I^{n_2} \to A_1\otimes_\Z A_2$ satisfies 
$$\mu_{\ast}(\vv^{n_1+n_2})=\mu_{\ast}(\vv^{n_1}\cup \vv^{n_2}) = {\mu_1}_\ast(\vv^{n_1})\cup {\mu_2}_\ast(\vv^{n_2}) = \alpha_1\cup \alpha_2.$$

\begin{lemma}\label{degreeone1} A degree one cohomology class $\alpha\in H^1(\pi\times \pi, A)$ is essential if and only if it is a zero-divisor.
\end{lemma}

The proof will be postponed until we have prepared the necessary algebraic techniques. 

\begin{lemma}\label{isom1}
Consider two left $\Z[\pi\times\pi]$-modules $M$ and $N$. 
Let $\tilde M$ and $\tilde N$ denote the left $\Z[\pi]$-module structures on $M$ and $N$ correspondingly via conjugation, 
i.e. $g\cdot m=gmg^{-1}$ and $g\cdot n=gng^{-1}$ for $g\in \pi$ and $m\in \tilde M$, $n\in \tilde N$. 
Let 
\begin{equation}\label{isom2}
\Phi: \Hom_{\zpp}(\zp\otimes_\Z M, N) \to \Hom_\zp(\tilde M, \tilde N)
\end{equation}
be the map which associates with any $\zpp$-homomorphism $f: \zp\otimes_\Z M \to N$ its restriction 
$f\, |\, e\otimes M$
onto $$M= e\otimes_\Z M\subset \zp\otimes_\Z M$$ 
where $e\in \pi$ is the unit element. Then $\Phi$ is an isomorphism.
\end{lemma}
\begin{proof} 
The inverse map 
\begin{equation}\label{isom3}
\Psi:  \Hom_\zp(\tilde M, \tilde N)\to \Hom_{\zpp}(\zp\otimes_\Z M, N)
\end{equation}
can be defined as follows. Given $\phi\in \Hom_\zp(\tilde M, \tilde N)$ let $\hat\phi: \zp\otimes_\Z M\to N$ be defined by 
$$\hat\phi(g\otimes m) =g\phi(g^{-1}m) =\phi(mg^{-1})g$$  for $g\in \pi$ and $m\in M$. For $a, b\in \pi$ we have 
\begin{eqnarray*}
\hat \phi(agb^{-1}\otimes amb^{-1}) &=& agb^{-1} \phi(bg^{-1}a^{-1}\cdot amb^{-1}) \\
&=& a\hat \phi(g\otimes m)b^{-1}
\end{eqnarray*}
which shows that $\hat \phi$ is a $\zpp$-homomorphism. We set $\Psi(\phi)=\hat \phi$. One checks directly that $\Psi$ and $\Phi$ are mutually inverse. 
\end{proof}

As the next step we prove the following generalisation of the previous lemma.

\begin{lemma}\label{isom4}
For two left $\Z[\pi\times\pi]$-modules $M$ and $N$ and any $i\ge 0$ the 
map 
\begin{equation}\label{isom5}
\Phi: \Ext^i_{\zpp}(\zp\otimes_\Z M, N) \to \Ext^i_\zp(\tilde M, \tilde N)
\end{equation}
is an isomorphism. The map $\Phi$ acts by first restricting the $\zpp$-module structure to the conjugate action of $\zp$ (where $\pi=\Delta_\pi\subset \pi\times \pi$ is the diagonal subgroup) and secondly by taking the restriction on the  $\zp$-submodule $\tilde M=e\otimes_\Z M\subset \zp\otimes_\Z M$. 
\end{lemma}

\begin{proof} Consider an injective resolution $0\to N\to J_0\to J_1\to \dots$ of $N$ over $\zpp$; we may use it to compute 
$\Ext^i_{\zpp}(\zp\otimes_\Z M, N)$. 
By Lemma \ref{isom1}, we have an isomorphism
$$\Phi: \Hom_\zpp(\zp\otimes_\Z M, J_i)\to \Hom_\zp(\tilde M, \tilde J_i), \quad i=0, 1, \dots$$
The statement of Lemma \ref{isom4} follows once we show that each module $\tilde J_i$ is injective with respect to the conjugate action of $\zp$. 

Consider an injective $\zpp$-module $J$, two $\zp$-modules $X\subset Y$ and a $\zp$-homo\-morphism $f: X\to \tilde J$ which needs to be extended onto $Y$. Note that $\zpp$ is free when viewed as a right $\zp$-module where the right action is given by $(g, h)\cdot k = (gk, hk)$ for $(g, h)\in \pi\times \pi$ and $k\in \pi$. Hence we obtain the $\zpp$-modules 
$$\zpp\otimes_{\zp} X \stackrel{\subset}\to \zpp\otimes_{\zp} Y$$
and the homomorphism $f: X\to \tilde J$ determines 
$$f': \zpp\otimes_{\zp} X\to  J$$
by the formula: $f'((g, h)\otimes x)= gf(x)h^{-1}$, where $g,h\in \pi$ and $x\in X$. It is obvious that $f'$ is well-defined and is a $\zpp$-homomorphism. 
Since $J$ is $\zpp$-injective, there is a $\zpp$-extension $$f'': \zpp\otimes_{\zp} Y\to J.$$ The restriction of $f''$ onto $Y= (e, e)\otimes Y$ is a $\zp$-homomorphism $Y\to \tilde J$ extending $f$. Hence $\tilde J$ is injective. This completes the proof.
\end{proof}

\begin{proof}[Proof of Lemma \ref{degreeone1}] Consider the short exact sequence of $\pi\times \pi$-modules
$$0\to I\to \Z[\pi]\stackrel{\epsilon}\to \Z\to 0$$
and the associated long exact sequence
$$
\dots \to \Hom_{\zpp}(I, A) \stackrel{\delta}\to \Ext^1_{\zpp}(\Z, A) 
\stackrel{\epsilon^\ast}\to  
\Ext^1_{\zpp}(\zp, A)\to ...$$
The condition that $\alpha\in H^1(\pi\times \pi, A)=\Ext^1_{\zpp}(\Z, A)$ is essential is equivalent to the requirement that $\alpha$ lies in the image of $\delta$. By exactness, it is equivalent to $\epsilon^\ast(\alpha)=0\in \Ext^1_{\zpp}(\zp, A)$. Consider the commutative diagram
\begin{equation}
\begin{array}{ccc}
\Ext^1_\zpp (\Z, A) & \stackrel{\epsilon^\ast}\to & \Ext^1_{\zpp}(\zp, A)\\ \\
\downarrow = && \simeq \downarrow \Phi\\ \\
H^1(\pi\times\pi, A)& \stackrel{\Delta^\ast}\to & \Ext^1_\zp(\Z, \tilde A) = H^1(\pi, \tilde A).
\end{array}
\end{equation}
Here $\Delta: \pi\to \pi\times \pi$ is the diagonal. The isomorphism $\Phi$ is given by Lemma \ref{isom4}. 
The commutativity of the diagram follows from the explicit description of $\Phi$. 
Thus we see that a cohomology class $\alpha\in H^1(\pi\times \pi, A)$ is essential if and only if $\Delta^\ast(\alpha)=0\in H^1(\pi, \tilde A)$, i.e. if $\alpha$ is a zero-divisor. 
\end{proof}

\begin{corollary}
For any $\pi\times \pi$ module $A$ one has an isomorphism
$$\Gamma: H^i(\pi\times\pi, \Hom_\Z(\zp, A)) \to H^i(\pi, \tilde A).$$
This isomorphism acts as follows: $$v\mapsto \omega_\ast(v\, |\, \pi), \quad v\in H^i(\pi\times\pi, \Hom_\Z(\zp, A))$$ where $\pi\subset \pi\times\pi$ is the diagonal subgroup and $\omega : \Hom_\Z(\zp, A)\to A$ is the homomorphism $\omega(f)=f(e)\in A$. The symbol $e$ denotes the unit element $e\in \pi$. 
\end{corollary}
\begin{proof}
Consider a free resolution $P_\ast$ of $\Z$ over $\zpp$. Then 
$$\Hom_{\zpp}(P_\ast, \Hom_\Z(\zp, A))\simeq \Hom_{\zpp}(P_\ast\otimes_\Z \zp, A)\simeq \Hom_\zp(\tilde P_\ast, \tilde A)$$
according to Lemma \ref{isom1}. Our statement now follows since $\tilde P_\ast$ is a free resolution of $\Z$ over $\zp$. 
\end{proof}

\section{The case of an abelian group}\label{abel}

Throughout this section we shall assume that the group $\pi$ is abelian. 
We shall fully describe the essential cohomology classes in $H^n(\pi\times \pi, A)$. 

First, we note that it makes sense to impose an additional condition on the $\zpp$-module $A$. 

For a $\zpp$-module $B$ let $B'\subset B$ denote the submodule 
$B'=\{b\in B; gb=bg\, \mbox{for any}\, g\in \pi\}$. Any $\zpp$-homomorphism $\mu: A\to B$ restricts to a homomorphism $\mu: A'\to B'$. 

Clearly, $I'=I$ and similarly $(I^n)'=I^n$. Hence any $\zpp$-homomorphism $\mu: I^n\to A$ takes values in the submodule $A'\subset A$. Hence discussing 
essential cohomology classes $\alpha\in H^n(\pi\times \pi, A)$ we may assume that $A'=A$. 

Consider the map 
\begin{equation}\label{phi}
\phi: \pi\times \pi\to \pi, \quad \mbox{where}\quad \phi(x, y) =xy^{-1}.
\end{equation}
It is a group homomorphism (since $\pi$ is abelian). Besides, let $A$ be a $\zpp$-module with $A'=A$. Then there exists a unique $\zp$-module $B$ such that $A=\phi^\ast(B)$. 

\begin{theorem}\label{abel} Assume that the group $\pi$ is abelian. Let $B$ be a $\zp$-module and let $\alpha\in H^n(\pi\times\pi, \phi^\ast(B))$ be a cohomology class. Then $\alpha$ is essential if and only if $\alpha=\phi^\ast(\beta)$ for some $\beta\in H^n(\pi, B)$. 
\end{theorem}

It follows from Theorem \ref{abel} that there are no nonzero essential cohomology classes $\alpha\in H^n(\pi\times \pi, A)$ with $n>\cd(\pi)$. 
Moreover, we see that if $\cd(\pi) <n\le  \cd(\pi\times \pi)$ then
any cohomology class $\alpha\in H^n(\pi\times \pi, A)$ is a zero-divisor which is not essential. 

\begin{proof} Assume that $\alpha\in H^n(\pi\times\pi, \phi^\ast(B))$ is such that $\alpha= \phi^\ast(\beta)$ where $\beta\in H^n(\pi, B)$. 
We want to show that $\alpha$ is essential. By Theorem \ref{thm3} there exists a $\zp$-homomorphism $\mu: I^n\to B$ such that $\mu_\ast(\vp^n)=\beta$ where $\vp\in H^1(\pi, I)$ is the Berstein - Schwarz class. 
Note that $\phi^\ast(I)=I$ and \begin{equation}\label{equal}
\phi^\ast(\vp)=\vv\end{equation} 
where $\vv\in H^1(\pi\times \pi, I)$ is the canonical class. To prove (\ref{equal}) we consider two subgroups 
$G_1=\pi\times 1\subset \pi\times \pi$ and $G_2=\Delta_\pi\subset \pi\times \pi$ and since $\pi\times \pi \simeq G_1\times G_2$ we can view the
Eilenberg-MacLane space 
$K(\pi\times\pi, 1)$ as the product $K(G_1, 1)\times K(G_2, 1)$. 
The restriction of the classes 
$\phi^\ast(\vp)$ and $\vv$ onto $G_1$ coincide (as follows from (\ref{v2b}) and from the definition of $\phi$). 
On the other hand, the restriction of the classes 
$\phi^\ast(\vp)$ and $\vv$ onto $G_2$ are trivial (as follows from (\ref{can1}) and from the definition of $\phi$). 
Now the equality (\ref{equal}) follows from the fact that the inclusion $K(G_1, 1)\vee K(G_2, 1)\to K(G_1, 1)\times K(G_2, 1)$ induces a monomorphism on 1-dimensional cohomology with any coefficients. 

Consider the commutative diagram
\begin{equation}\label{diag1}
\begin{array}{ccc}
H^n(\pi, I^n) & \stackrel{\mu_\ast}\longrightarrow & H^n(\pi, B)\\ \\
\phi^\ast \downarrow &&\downarrow \phi^\ast\\ \\
H^n(\pi\times\pi, I^n) & \stackrel{\mu_\ast}\longrightarrow & H^n(\pi\times \pi, \phi^\ast(B)).
\end{array}
\end{equation}
The upper left group contains the power $\vp^n$ of the Berstein - Schwarz class which is mapped onto $\beta=\mu_\ast(\vp^n)$ and $\phi^\ast(\beta)=\alpha$. Moving in the other direction we find $\alpha=\mu_\ast(\phi^\ast(\vp^n))= \mu_\ast(\vv^n)$, i.e. $\alpha$ is essential. 

To prove the inverse statement, assume that a cohomology class $$\alpha\in H^n(\pi\times \pi, \phi^\ast(B))$$ is essential, i.e. 
$\alpha= \mu_\ast(\vv^n)$
for a $\zpp$-homomorphism $\mu: I^n \to \phi^\ast(B)$. We may also view $\mu$ as a $\zp$-homomorphism $I^n\to B$ which leads to the commutative diagram (\ref{diag1}). Using (\ref{equal}) we find that $\phi^\ast(\mu_\ast(\vp^n)) =\mu_\ast( \phi^\ast(\vp^n))= \mu_\ast(\vv^n)=\alpha.$
Hence we see that $\alpha=\phi^\ast(\beta)$ where $\beta= \mu_\ast(\vp^n)$.
\end{proof}

If we wish to be specific, let $\pi=\Z^N$ and consider the trivial coefficient system $A=\Z$. Then $N$ is the highest dimension in which essential cohomology classes 
$$\alpha\in H^N(\Z^N \times \Z^N;\Z)$$ exist. Due to Theorem \ref{abel}, all $N$-dimensional essential cohomology classes are integral multiples of a single class which we are going to describe. 

Denote by $x_1, \dots, x_N\in H^1(\Z^N,\Z)$ a set of generators. Then each class 
$$\alpha_i = x_i\otimes 1 - 1\otimes x_i\in H^1(\Z^N\times \Z^N; \Z), \quad i=1, \dots, N,$$
is a zero-divisor and hence is essential by Lemma \ref{degreeone1}. Their product  
$$\alpha =\alpha_1\cup \alpha_2\cup\dots\cup \alpha_N \in H^N(\Z^N\times \Z^N;\Z)$$
is essential as a product of essential classes. We may write $\alpha$ as the sum of $2^N$ terms
$$\alpha= (-1)^N \cdot \sum_K\, (-1)^{|K|} \cdot x_K\otimes x_{K^c}$$
where $K\subset \{1, 2, \dots, N\}$ runs over all subsets of the index set and $K^c$ denotes the complement of $K$. 
For  $K=\{i_1, i_2, \dots, i_k\}$ with $i_1<i_2<\dots< i_k$ the symbol $x_K$ stands for the product
$x_{i_1}x_{i_2}\dots x_{i_k}$.

\section{The spectral sequence}

In this and in the subsequent sections we abandon the assumption that $\pi$ is abelian and return to the general case, i.e. we consider an arbitrary discrete group $\pi$. 

Let $A$ be a left $\zpp$-module. 
We shall describe an exact couple and a spectral sequence which will allow us to find a sequence of obstructions for a cohomology class $\alpha\in H^\ast(\pi\times \pi,A)$ to be essential. 

We introduce the following notations: 
$$E_0^{rs} = \Ext_{\zpp}^r({\zp}\otimes_{\Z} I^s, A) \quad \mbox{and}\quad D_0^{rs} = \Ext_\zpp^r(I^s, A).$$
The long exact sequence associated to the short exact sequence (\ref{s}) can be written in the form
$$\dots \to E_0^{rs} \stackrel{k_0}\to D_0^{r, s+1}\stackrel{i_0} \to D_0^{r+1, s} \stackrel{j_0}\to E_0^{r+1, s}\to \dots$$
Here $i_0: D_0^{rs}\to D_0^{r+1, s-1}$ is the connecting homomorphism 
\begin{equation}\label{ss}
\Ext_\zpp^r(I^s, A) \to \Ext_\zpp^{r+1}(I^{s-1}, A)\end{equation}
corresponding to the exact sequence $(\ref{s})$. Note that 
\begin{equation}
D_0^{n, 0} = H^n(\pi\times \pi, A), \quad \mbox{and}\quad D_0^{0, n}= \Hom_\zpp(I^n, A).
\end{equation}

\begin{lemma}\label{lmdde}
The set of essential cohomology classes in $H^n(\pi\times \pi, A)$
coincides with the image of the composition of $n$ maps $i_0$:
\begin{equation}
\Hom_\zpp(I^n, A)= D_0^{0, n} \stackrel{i_0}\to D_0^{1, n-1} \stackrel{i_0}\to \dots \stackrel{i_0}\to D_0^{n, 0}= H^n(\pi\times \pi, A).
\end{equation}\label{dde}
\end{lemma}
\begin{proof} Applying the technique described in \cite{HS}, chapter IV, \S 9, we obtain that the image of a homomorphism 
$f\in \Hom_\zpp(I^n, A)$ under the composition $i_0^n$ is an element of $\Ext^n_\zpp(\Z, A)$ represented by the exact sequence
$$0\to A\to X_f \to \zp\otimes_\Z I^{n-2} \to \zp\otimes_\Z I^{n-3} \to\dots\to \zp\to \Z\to 0$$
where $X_f$ appears in the push-out diagram 
\begin{equation*}
\begin{array}{ccc} 
I^n & \to & \zp\otimes_\Z I^{n-1}\\
{\, \, }\downarrow f && \downarrow\\
A & \to & X_f.
\end{array}
\end{equation*}
Using Lemma \ref{vvrepresented} we see that the same exact sequence represents the element $f_\ast(\vv^n)$. 
\end{proof}

A different proof of Lemma \ref{lmdde} will be given later in this section. 

The exact sequences (\ref{dde}) can be organised into a bigraded exact couple as follows. Denote
\begin{equation}
E_0= \bigoplus_{r, s\ge 0} E^{rs}_0 = \bigoplus_{r, s\ge 0} \Ext^r_{\pi\times \pi}(\zp\otimes_\Z I^s, A), 
\end{equation}
and 
\begin{equation}
D_0 =  \bigoplus_{r, s\ge 0} D^{rs}_0 =\bigoplus_{r, s\ge 0} \Ext^r_{\pi\times \pi}( I^s, A).
\end{equation}
The exact sequence (\ref{dde}) becomes an exact couple
$$\begin{array}{ccccc}
D_0 && \stackrel{i_0}\longrightarrow & & D_0\\ \\
&k_0 \nwarrow &&\swarrow j_0&\\ \\
&&E_0&&
\end{array}
$$
Here the homomorphism $i_0$ has bidegree $(1, -1)$,  the homomorphism $k_0$ has bidegree (0,1), and  the homomorphism $j_0$ has bidegree $(0,0)$.
%
%
%
%
Applying the general formalism of exact couples, we may construct the $p$-th derived couple
$$\begin{array}{ccccc}
D_p && \stackrel{i_p}\longrightarrow & & D_p\\ \\
&k_p \nwarrow &&\swarrow j_p&\\ \\
&&E_p&&
\end{array}
$$
where $p=0, 1, \dots$.
The module $D^{rs}_p$ is defined as 
$$D^{rs}_p=\im[i_{p-1}:D^{r-1,s+1}_{p-1}\to D^{rs}_{p-1}] = \im [i_0\circ \dots\circ i_0: D_0^{r-p, s+p}\to D_0^{r, s}].$$
 and 
 $$E_p^{\ast, \ast}=H(E_{p-1}^{\ast, \ast}, d_{p-1})$$ is the homology of the previous term with respect to the differential 
 $d_{p-1}=j_{p-1}\circ k_{p-1}.$
The degrees are as follows:

\begin{eqnarray*}
\deg j_p&=& (-p, p),\\
\deg i_p &=& (1, -1), \\
\deg k_p &=& (0, 1),\\
\deg d_p &=& (-p, p+1).
\end{eqnarray*}

Using this spectral sequence we can express the set of essential classes as follows:

\begin{corollary}\label{essential}
The group $D_n^{n, 0}\subset H^n(\pi\times\pi, A)=D_0^{n, 0}$ coincides with the set of all essential cohomology classes in $H^n(\pi\times\pi, A)$.
\end{corollary}
\begin{proof}
This is equivalent to Lemma \ref{lmdde}. 
\end{proof}

We want to express the homomorphism 
$i_0: D_0^{r, s+1}\to D_0^{r+1, s}$ with $s\ge 0$
through the canonical class $\vv\in H^1(\pi\times \pi, I)$. 
This will be used to give a different proof of Lemma \ref{dde} and will have some other interesting applications. 
According to the definition, $i_0$ is the connecting homomorphism 
$$i_0: \, \Ext_\zpp^r(I^{s+1}, A) \to \Ext_\zpp^{r+1}(I^{s}, A)$$ corresponding to the short exact sequence (\ref{s}).
Note that $$D_0^{r,s} =\Ext^r_\zpp(I^s, A)= H^r(\pi\times\pi, \Hom_\Z(I^s, A)),$$ see 
\cite{Brown}, chapter III, Proposition 2.2. Under this identification $i_0$ turns into the Bockstein homomorphism 
\begin{equation}\label{bockstein}
\beta: H^r(\pi \times \pi;\mathrm{Hom}_{\Z}(I^{s+1},A)) \to H^{r+1}(\pi \times \pi;\Hom_\Z(I^{s},A))\end{equation}
corresponding to the short exact sequence of $\zpp$-modules
\begin{equation}\label{dual}0 \to \mathrm{Hom}_\Z(I^{s},A) \to \mathrm{Hom}_\Z(\zp \otimes_\Z I^{s},A) \to \mathrm{Hom}_\Z(I^{s+1},A) \to 0.
\end{equation}
The sequence (\ref{dual}) is obtained by applying the functor $\Hom_\Z(\, \cdot \, , A)$ to the exact sequence (\ref{s}) (note that (\ref{s}) splits over $\Z$).  

Let 
$$\ev: \, I\otimes_\Z \Hom_\Z(I^{s+1}, A) \to \Hom_\Z(I^{s}, A)$$
denote the homomorphism given by
$$x_0\otimes f \, \mapsto \, \left(x_1\otimes \dots\otimes x_s\mapsto f(x_0\otimes x_1\otimes \dots \otimes x_s)\right)$$
for $f\in \Hom_\Z(I^s, A)$ and $x_i\in I$ for $i=0, 1, \dots, s$.

\begin{proposition}\label{propbockstein}
For any cohomology class $u \in H^r(\pi \times \pi,\mathrm{Hom}_\Z(I^{s+1},A))$ one has
\begin{equation}\label{dual1}
\beta(u) = -\ev_*(\mathfrak{v}\cup u), \end{equation}
 where $\vv\in H^1(\pi\times\pi, I)$ denotes the canonical class.
\end{proposition}

\begin{proof}
Using \cite{Brown}, chapter V, property (3.3) and Lemma 5 from \cite{CF}, we obtain $ \delta(u)= \mathfrak{v}\cup u,$ where 
$$\delta:H^r(\pi \times \pi;\mathrm{Hom}_\Z(I^{s+1},A)) \to H^{r+1}(\pi \times \pi;I \otimes \mathrm{Hom}_\Z(I^{s},A))$$
is the Bockstein homomorphism associated with the short exact coefficient sequence
$$0 \to I \otimes_\Z \mathrm{Hom}_{\Z}(I^{s+1},A) \to \Z[\pi] \otimes_\Z \mathrm{Hom}_{\Z}(I^{s+1},A)\stackrel{\epsilon\otimes \mathrm{id}}\to \mathrm{Hom}_{\Z}(I^{s+1},A) \to 0 \ . $$
The latter sequence is obtained by tensoring (\ref{can}) with $\Hom_\Z(I^{s+1}, A)$ over $\Z$. 
To prove Proposition \ref{propbockstein} it is enough to show that 
$\beta {=} -\ev_* \circ \delta. $ Having this goal in mind, we
denote by $$F: \Z[\pi] \otimes_\Z \mathrm{Hom}_{\Z}(I^{s+1},A) \to \mathrm{Hom}_{\Z}(\Z[\pi]\otimes_\Z I^{s},A)$$ the homomorphism which extends $\Z$-linearly the following map
$$F(x\otimes f)(z \otimes y) = f((z-x)\otimes y) $$
for $x,z \in \pi, \ y \in I^{s}$ and $f\in \Hom_\Z(I^{s+1}, A).$
We compute: 
\begin{eqnarray*}	
&{}&F((g,h)\cdot (x \otimes f))(z \otimes y) = F(gxh^{-1}  \otimes (g,h)f)(z \otimes y) \\
&=& ((g,h)f)((z-gxh^{-1})\otimes y) = gf((g^{-1}zh-x)\otimes g^{-1}yh)h^{-1} \\
&=& g F(x\otimes f)(g^{-1}zh \otimes g^{-1}yh)h^{-1} = ((g,h)\cdot F(x\otimes f))(z \otimes y) \ .
\end{eqnarray*}
Hence 
$F$ is a $\Z[\pi \times \pi]$-homomorphism. Next we claim that the following diagram with exact rows commutes:
$$ \begin{array}{ccccccccc}
0 & \to & I \otimes_{\Z} \mathrm{Hom}_{\Z}(I^{s+1},A) & \stackrel{i \otimes \mathrm{id}}{\longrightarrow} & \Z[\pi]\otimes_{\Z} \mathrm{Hom}_{\Z}(I^{s+1},A) &\stackrel{\epsilon\otimes \mathrm{id}}{\longrightarrow} & \mathrm{Hom}_{\Z}(I^{s+1},A) & \to & 0 \\
& & \downarrow \mathbf{ev} & & \downarrow F & & \downarrow \mathrm{id} & & \\
0 & \to &  \mathrm{Hom}_{\Z}(I^s,A) & \stackrel{-\epsilon^*}{\longrightarrow} &  \mathrm{Hom}_{\Z}(\Z[\pi]\otimes_{\Z} I^s,A) &\stackrel{i^*}{\longrightarrow} & \mathrm{Hom}_{\Z}(I^{s+1},A) & \to & 0 
\end{array}
$$
%
%
Indeed, we compute for $g,h \in \pi$, $y \in I^{s}$ and $f \in \mathrm{Hom}_\Z(I^{s+1},A)$,
$$(-\epsilon^*\circ \ev)((g-1)\otimes f)(h \otimes y) = -\epsilon(h) f((g-1)\otimes y) = -f((g-1)\otimes y)$$
and 
\begin{eqnarray*}
&{}&(F\circ (i \otimes \mathrm{id})((g-1)\otimes f)(h\otimes y)  \\ &=& (F(g \otimes f))(h \otimes y) - (F(1 \otimes f))(h\otimes y) \\ 
&=&f((h-g)\otimes y)-f((h-1)\otimes y) \\ &=& -f((g-1)\otimes y) \ .
\end{eqnarray*}
Hence, we see that the left square in the above diagram commutes. We further observe that for all $g,h \in \pi$, $y \in I^{s}$ and 
$f \in \mathrm{Hom}_{\Z}(I^{s+1},A)$ one has
\begin{eqnarray*}
&{}&((i^*\circ F)(g \otimes f))((h-1)\otimes y) = F(g \otimes f)(h\otimes y) - F(g \otimes f)(1\otimes y)  \\ &=& f((h-g)\otimes y)-f((1-g)\otimes y) = f((h-1)\otimes y) \\ 
&=& \epsilon(g)f((h-1)\otimes y) = ((\epsilon\otimes \mathrm{id})(g\otimes f))((h-1)\otimes y),
\end{eqnarray*}
and hence the right square of the diagram commutes as well. 

The commutativity of the above diagram implies that the Bockstein homomorphisms satisfy 
$$\ev_*\circ \delta = \beta' \circ {\rm {id}} = \beta' \ , $$
where $\beta'$ denotes the Bockstein homomorphism of the bottom row exact sequence. Since this sequence coincides with the sequence associated with $\beta$ up to a sign change in the first map, one derives from the snake lemma that $\beta' = -\beta$. This completes the proof. 
\end{proof}

\begin{corollary}\label{coress} Let $\alpha\in H^n(\pi\times\pi, A)$ be a cohomology class and let $k=1, 2, \dots, n-1$ be an integer. Then the following conditions are
equivalent: 
\begin{enumerate}
\item $\alpha$ lies in $D_k^{n, 0}$. 

\item $\alpha = \psi_\ast(\vv^k\cup u)$ for a cohomology class $u\in H^{n-k}(\pi\times\pi, \Hom_\Z(I^k, A))$ where 
$$\psi: I^k \otimes \Hom_\Z(I^k, A) \to A$$
is the coefficient pairing  
\begin{equation}\label{reverse}
\psi(x_1\otimes \dots \otimes x_k \otimes f) =   f(x_k\otimes x_{k-1}\otimes \dots \otimes x_1).\end{equation}
\end{enumerate}
\end{corollary}
\begin{proof} 
The condition $\alpha\in D_k^{n,0}$ means that $\alpha= i_0^k(u)$ for some $u\in D_0^{n-k, k}$. We know that 
$D_0^{n-k, k} = \Ext_\zpp^{n-k}(I^k, A) = H^{n-k}(\pi\times \pi, \Hom_\Z(I^k, A)).$
Our statement follows by applying iteratively Proposition \ref{propbockstein}. 

%
\end{proof}

We may use Proposition \ref{propbockstein} to give another proof of Lemma {\ref{lmdde}}. By Corollary \ref{coress}, classes $\alpha\in D_n^{n,0}$ are characterised by the property $\alpha= \vv^n\cup u$ where the cup product is given with respect to the pairing 
$I^n\otimes \Hom_\Z(I^n, A)\to A$ given by the formula (\ref{reverse}) for some $u\in H^0(\pi\times\pi, \Hom_\Z(I^n, A))=\Hom_\zpp(I^n, A)$. 
Thus $u$ is a $\zpp$-homomorphism $I^n\to A$ and applying the definition of the cup product (see \cite{Brown}, chapter V, \S 3) we see that 
$\alpha\in D_n^{n,0}$ if and only if $\alpha=\phi_\ast(\vv^n)$ for a $\zpp$-homomorphism $\phi: I^n\to A$. 

\begin{corollary}\label{cor65}
If for some $\zpp$-module $A$ and for an integer $k$ the module $D_k^{n, 0}$ is nonzero then $\tc(K(\pi, 1))\ge k+1$. 
\end{corollary}
\begin{proof}
Using Corollary \ref{coress} we see that $D_k^{n, 0}\not=0$ then for $\alpha\in D_k^{n, 0}$, $\alpha\not =0$, we have $\alpha = \psi_\ast(\vv^k\cup u)$ and hence $\vv^k\not=0$. Since $\vv$ is a zero-divisor, our statement follows from \cite{Finv}, Corollary 4.40. 
\end{proof}
Using the spectral sequence we may describe a complete set of $n$ obstructions for a cohomology class $\alpha\in H^n(\pi\times\pi, A)=D_0^{n, 0}$ to be essential. 
We shall apply Lemma \ref{lmdde} and act inductively. The class $\alpha$ is essential if it lies in the image of the composition of $n$ maps $i_0$. For this to happen we first need to guarantee that $\alpha$ lies in the image of the last map $i_0: D_0^{n-1, 1}\to D_0^{n, 0}$. Because of the exact sequence
$$\dots \to D_0^{n-1, 1}\, \stackrel{i_0}\to \, D_0^{n, 0}\, \stackrel{j_0}\to \, E_0^{n, 0} \, \stackrel{k_0}\to \, \dots$$
we see that $\alpha$ lies in the image of $i_0$ if and only if 
\begin{equation}\label{j0}
j_0(\alpha)=0\in E_0^{n, 0}.
\end{equation} 
We have the commutative diagram
$$
\begin{array}{ccc}
\Ext_\zpp^n(\Z, A) & \stackrel{ j_0}\to & \Ext_\zpp^n(\zp, A)\\ \\
\downarrow = && \Phi\downarrow \simeq\\ \\
H^n(\pi\times \pi, A) & \stackrel{r^\ast}\to & H^n(\pi, \tilde A).
\end{array}
$$
Here $\Phi$ is the isomorphism of Lemma \ref{isom4} and $j_0=\epsilon^\ast$ is the homomorphism induced by the augmentation; the homomorphism 
$r^\ast$ is induced by the inclusion $r: \pi\to \pi\times \pi$ of the diagonal subgroup. We see that the class $\alpha$ lies in the image of $i_0$ if and only if it is a zero divisor, i.e. $r^\ast(\alpha) =0$. 

To describe the second obstruction let us assume that $\alpha\in H^n(\pi\times\pi, A)$ is a zero-divisor, i.e. (\ref{j0}) is satisfied. Then 
$\alpha \in D_1^{n,0}$.  One has $\alpha \in D_2^{n,0}$ if and only if 
\begin{equation}\label{j1}
j_1(\alpha)=0 \in E_1^{n-1, 1}.
\end{equation} 
This follows from the exact sequence
$$\dots \to D_1^{n-1, 1}\, \stackrel{i_1}\to \, D_1^{n, 0}\, \stackrel{j_1}\to \, E_1^{n-1, 1} \, \stackrel{k_1}\to D_1^{n-1, 2}\to \, \dots$$
where $i_1$ is the restriction of $i_0$ onto $D_1\subset D_0$. 

Continuing these arguments and using the exact sequences
$$\dots \to D_p^{n-1, 1}\, \stackrel{i_p}\to \, D_p^{n, 0}\, \stackrel{j_p}\to \, E_p^{n-p, p} \, \stackrel{k_p}\to D_p^{n-p, p+1}\to \, \dots$$
we arrive at the following conclusion:
\begin{corollary}\label{cor66} Let $k$ and $n$ be integers with $0<k\le n$. 
\begin{enumerate}
\item[(1)] A cohomology class $\alpha\in H^n(\pi\times\pi, A)$ lies in the group \newline $D_k^{n, 0} = \im [i_0^k: D_0^{n-k, k}\to D_0^{n, 0}]$ if and only if the following 
$k$ obstructions 
\begin{eqnarray}\label{jk}
j_s(\alpha)\, \in \, E_s^{n-s, s}, \quad \mbox{where}\quad s = 0, 1, \dots, k-1,
\end{eqnarray}
vanish. 

\item[(2)]  The condition $j_0(\alpha)=0$ is equivalent for $\alpha$ to be a zero-divisor. 

\item[(3)]  Each obstruction $j_s(\alpha)$ is defined once the previous
obstruction $j_{s-1}(\alpha)$ vanishes. 

\item[(4)]  The triviality of all obstructions 
$j_0(\alpha), j_1(\alpha), \dots, j_{n-1}(\alpha)$
is necessary and sufficient for the cohomology class $\alpha$ to be essential. 
\end{enumerate}
\end{corollary}
Figure \ref{eterm} shows the locations of the obstructions $j_k(\alpha)\in E_k^{n-k, k}$. 
\begin{figure}[h]
\begin{center}
{\includegraphics[width=0.45 \textwidth]{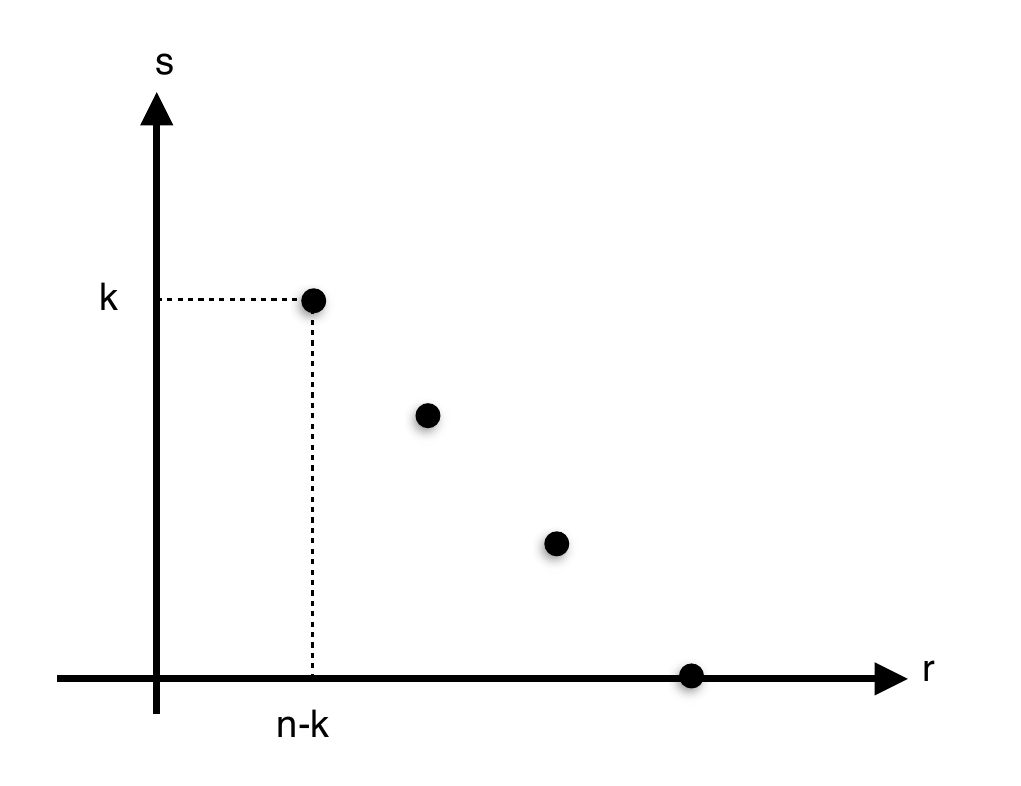}}
\end{center}
\caption{The groups in the $E$-term of the spectral sequence containing 
the obstructions $j_k(\alpha)$.}\label{eterm}
\end{figure}

%
%
%
%
\section{Computing the term $E_0^{r, s}$ for $s\ge 1$}

In this section we compute the initial term $E_0^{r, s}$ of the spectral sequence. 
Using Lemma \ref{isom5} we find
\begin{eqnarray}
E_0^{r, s} &=& \Ext_\zpp^r(\zp\otimes_\Z I^s, A) \nonumber \\
&\simeq & \Ext_\zp^r(\tilde I^s, \tilde A)\label{step1}
\end{eqnarray}
where in the second line the tilde $\quad \tilde{}\quad$ above the corresponding modules means that we consider these modules with respect to the conjugation action,
i.e. $g\cdot a=gag^{-1}$ for $a\in A$ and $g\in \pi$. We exploit below the simple structure of the module $\tilde I^s$ where $s\ge 1$ and compute explicitly the $E_0$-term. 

For $s\ge 1$ consider the action of $\pi$ on the Cartesian power $\pi^s=\pi\times\pi\times \dots\times \pi$ ($s$ times) via conjugation, i.e. 
$g\cdot(g_1, \dots, g_s) = (gg_1g^{-1}, \dots, gg_sg^{-1})$ for $g, g_1, \dots g_s\in \pi$. 
The orbits of this action are joint conjugacy classes of $s$-tuples of elements of $\pi$. 
We denote by $\C_{\pi^s}$ the set of orbits and let $\C'_{\pi^s}\subset \C_{\pi^s}$ denote the set of orbits of nontrivial elements, i.e. such that  
$g_i\not=1$ for all $i=1, \dots, s$. 

Let $C\in \cps$ be an orbit. The isotropy subgroup $N_C\subset \pi$ of an $s$-tuple $(g_1, \dots, g_s)\in C\subset \pi^s$ is the intersection of the centralisers of the elements $g_1, \dots, g_s$. The subgroup $N_C$, viewed up to conjugation, depends only on the orbit $C$. 

\begin{theorem} \label{thm5} For any left $\zpp$-module $A$ and for integers $r\ge 0$ and $s\ge 1$ one has
\begin{equation}\label{prod3}
E_0^{r, s} \simeq \prod_{C\in \cps} H^r(N_C, A\, |\, {N_C}).
\end{equation}
Here $A\, |\, {N_C}$ denotes $A$ viewed as $\Z[N_C]$-module with $N_C\subset \pi=\Delta_\pi\subset \pi\times\pi$. 
\end{theorem}
\begin{proof} For any $C\in \cps$ consider the set $J_C\subset \tilde{I^s}$ generated over $\Z$ by the tensors of the form
\begin{equation}\label{tensor}
(g_1-1)\otimes\dots\otimes (g_s-1)\end{equation} 
for all $(g_1, \dots, g_s)\in C$. It is clear that $J_C$ is a $\zp$-submodule of $\tilde{I^s}$ (since we consider the conjugation action). 
Moreover, we observe that 
\begin{equation}\label{sum3}
\tilde{I^s} = \bigoplus_{C\in \cps} J_C. 
\end{equation}
Indeed, the elements $g-1$ with various $g\in\pi^\ast$ (where we denote $\pi^\ast= \pi-\{1\}$) form a free $\Z$-basis of $I$; therefore elements of the form 
$(g_1-1)\otimes\dots\otimes(g_s-1)$ with all possible $g_1, \dots, g_s\in \pi^\ast$ form a free $\Z$-basis of $I^s$. 
The formula (\ref{sum3}) is now obvious.

For $C\in \cps$ let $\Z[C]$ denote the free abelian group generated by $C$. Since $\pi$ acts on $C$, the group $\Z[C]$ is naturally a left $\zp$-module which is isomorphic to $J_C$ via the isomorphism
$$(g_1, \dots, g_s)\, \mapsto\, (g_1-1)\otimes\dots\otimes (g_s-1),$$
where $(g_1, \dots, g_s)\in C.$
For a left $\zp$-module $B$ we have
\begin{eqnarray*}\Hom_\zp(J_C, B) &=& \Hom_\zp(\Z[C], B|_{N_C})\\ &=& \Hom_{N_C}(\Z, B)\\ &=& H^0(N_C, B|_{N_C}).\end{eqnarray*}
Here we used the fact that the action of $\pi$ on $C$ is transitive and hence a $\zp$-homomorphism $f: \Z[C]\to B$ is uniquely determined by one of its values $f(c)$ where $c\in C$. 

Consider a free resolution
$$P_\ast: \quad \dots \to P_n\to P_{n-1}\to \dots\to P_0\to \Z\to 0$$
 of $\Z$ over $\Z[\pi]$.
Since $\Z[C]$ is free as an abelian group we have the exact sequence 
\begin{equation}\label{product2}
\dots \to \zc\otimes_\Z P_n \to \zc\otimes_\Z P_{n-1} \to \dots\to \zc\otimes_\Z P_0 \to \zc\to 0\end{equation}
of $\zp$-modules. 
It is easy to see that each module $\zc\otimes_\Z P_n $ (equipped with the diagonal action) is free as a $\Z[\pi]$-module 
(see \cite{Brown}, chapter III, Corollary 5.7). 
Thus we see that (\ref{product2}) is a free $\zp$-resolution of $\zc$ and we may use it to compute $\Ext_\zp^r(\zc, B)$. 
We have 
\begin{eqnarray*}
\Hom_\zp(\zc\otimes_\Z P_n,  B) =  \Hom_\zp(\zc, \Hom_\Z(P_n, B)) = 
\Hom_{\Z[N_C]}(P_n, B|{N_C}).
\end{eqnarray*}
Thus we see that the complex 
$\Hom_\zp(\zc\otimes P_\ast, B)$ which computes  $\Ext^r_\zp(J_C, B)$, 
coincides with the complex \begin{eqnarray}\label{coh3}\Hom_{\Z[N_C]}( P_\ast |_{N_C}, B | {N_C}).\end{eqnarray}
Since $P_n$ is free as a $\Z [N_C]$-module, we see that the cohomology of the complex (\ref{coh3}) equals 
$ H^r(N_C, B\, |\, {N_C})$. Thus we obtain isomorphisms
\begin{eqnarray}\label{final}
\Ext^r_\zp(J_C, B) \simeq \Ext^r_\zp(\Z[C], B)\simeq H^r(N_C; B|{N_C}).
\end{eqnarray}

Combining the isomorphisms (\ref{step1}), (\ref{sum3}) and (\ref{final}) we obtain the isomorphism (\ref{prod3}). 
\end{proof}

\begin{corollary}\label{corfinal}
Let $\pi$ be a discrete torsion-free group such that the centraliser of any nontrivial element $g\in \pi$, $g\not=1$ is infinite cyclic. 
Then $$E_0^{r,s}=0$$ for all $r>1$ and $s\ge 1$. 
\end{corollary}
\begin{proof}
Applying Theorem \ref{thm5} we see that each group $N_C$, where $C\in \cps$, is a subgroup of $\Z$ and hence it is either $\Z$ or trivial. The result now follows from (\ref{prod3}) since we assume that $r>1$. 
\end{proof}

\begin{figure}[h]
\begin{center}
{\includegraphics[width=0.45 \textwidth]{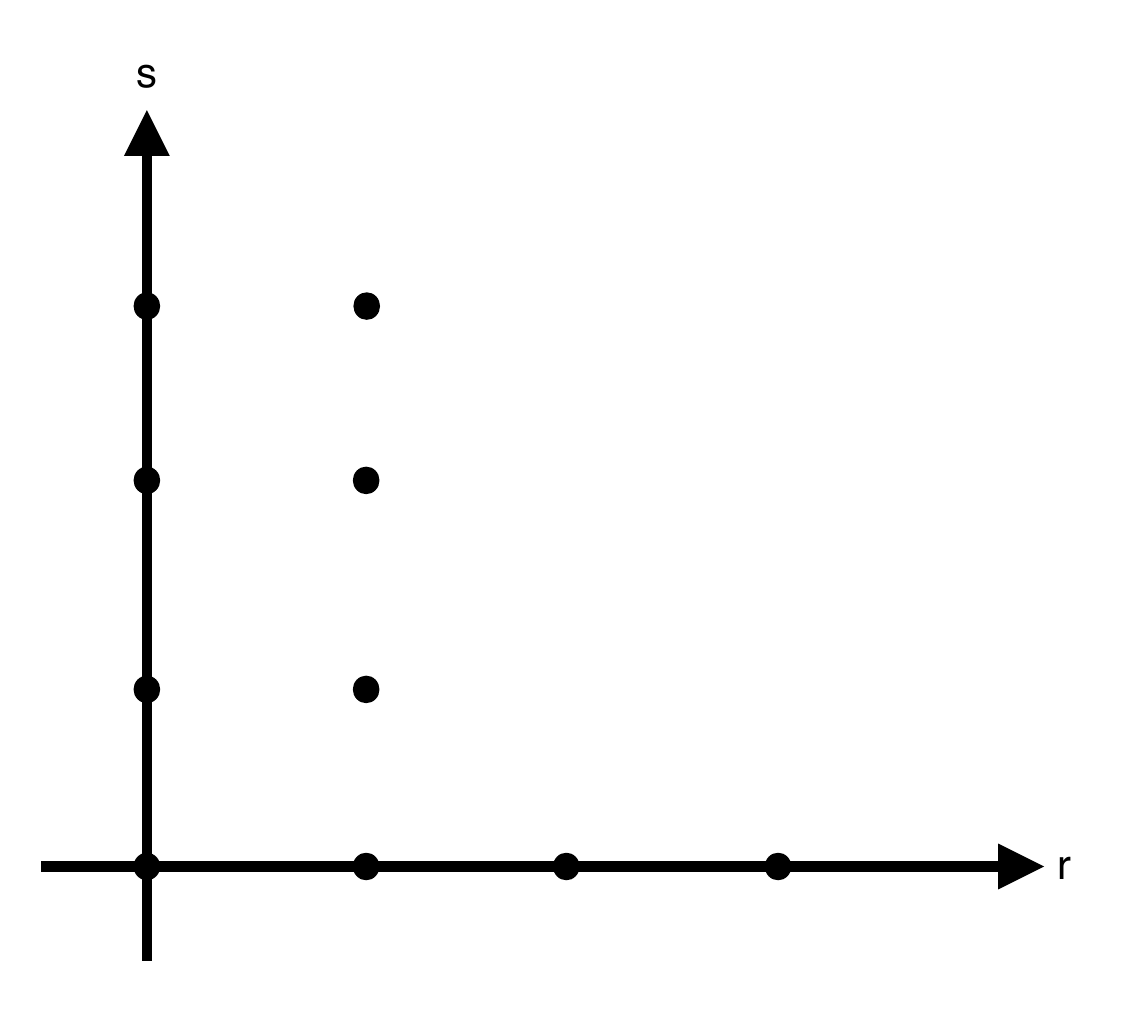}}
\end{center}
\caption{The nontrivial groups in the $E_0$-term of the spectral sequence.}\label{nonzero}
\end{figure}
Figure \ref{nonzero} shows potentially nontrivial groups in the $E_0$-term in the case when all centralisers of nontrivial elements are cyclic.

\section{Proofs of Theorems \ref{thm1}, \ref{thm11}, \ref{thm2} and \ref{cor4} }\label{proofs}

\begin{proof}[Proof of Theorem \ref{thm2}] Let $X=K(\pi, 1)$ where the group $\pi$ satisfies our assumption that the centraliser of any nonzero element is infinite cyclic. 
Let $\alpha\in H^n(X\times X, A)$ be a zero-divisor. Here $A$ is a local coefficient system over $X\times X$. By Corollary \ref{cor66}, statement (2), we have $j_0(\alpha)=0$. Besides, applying Corollary \ref{corfinal} we see that the obstructions $j_s(\alpha)\in E_s^{n-s, s}$ vanish for $s=1, 2, \dots, n-2$ since they lie in the trivial groups. Thus we obtain $$\alpha\in D_{n-1}^{n,0}.$$
Next we apply Corollary \ref{coress} which gives
\begin{eqnarray}\label{formula1}
\alpha = \psi_\ast(\vv^{n-1}\cup u)\end{eqnarray} 
for some $u\in H^{1}(\pi\times\pi, \Hom_\Z(I^{n-1}, A))$ where $\psi$ is given by (\ref{reverse}). 

To prove that $\wgt(\alpha)\ge n-1$ we observe that the canonical class $\vv$ has positive weight, $\wgt(\vv)\ge 1$, since it is a zero-divisor, see (\ref{can1}). 
Hence using (\ref{add}) we obtain $\wgt(\vv^{n-1}) \ge n-1$. Let $f: Y\to X\times X$ be a continuous map as in Definition \ref{defweight}. Then 
$$f^\ast(\alpha) = \psi_\ast(f^\ast(\vv^{n-1})\cup f^\ast(u))=0$$ since $f^\ast(\vv^{n-1})=0$. 
%
This completes the proof. 
\end{proof}
\begin{proof}[Proof of Theorem \ref{thm11}]
Suppose that we are in the situation of Theorem \ref{thm11}, i.e. let $X$ be an aspherical finite cell complex such whose fundamental group 
$\pi=\pi_1(X)$ has 
the properties (1) and (2). 
Denote $n=\cd(\pi\times \pi)$. We may find a local coefficient system $A$ over $X\times X$ and a nonzero cohomology class $\alpha\in H^n(X\times X, A)$. Since $n>\cd(\pi)$ we obtain that $\alpha$ is a zero-divisor. Next we apply Theorem \ref{thm2} which implies that the weight of $\alpha$ satisfies $\wgt(\alpha)\ge n-1$. Thus we obtain that $\tc(X)\ge n$. 

The inequality $\tc(X)\le n+1$ follows from the Eilenberg - Ganea theorem and general dimensional upper bound for the topological complexity $\tc(X) \le \dim(X\times X)+1$. 

Hence,
$\tc(X)$ is either $n$ or $n+1$. 
\end{proof}

\begin{proof}[Proof of Theorem \ref{cor4}] As in the proof of Theorem \ref{thm11} we may find a nonzero cohomology class 
$\alpha\in H^n(X\times X, A)$, where $n=\cd(\pi\times \pi)$, which is automatically a zero-divisor, since $n>\cd(\pi)$. Applying the arguments used in the 
proof of Theorem \ref{thm2} we find that $\alpha = \psi_\ast(\vv^{n-1}\cup u)$, see (\ref{formula1}), implying that $\vv^{n-1}\not=0$. 

\end{proof}

\begin{proof}[Proof of Theorem \ref{thm1}] Let $X$ be an aspherical finite cell complex with $\pi=\pi_1(X)$ hyperbolic. 
Then $\pi$ is torsion-free and 
we may assume that $\pi\not= 1$ since in the simply connected case our statement is obvious. 
The centraliser $$Z(g)=\{h\in \pi; hgh^{-1}=g\}$$ of any nontrivial element $g\in \pi$ is virtually cyclic, see \cite{BH}, Corollary 3.10 in chapter III. 
It is well known that any torsion-free virtually cyclic group is cyclic. Thus, we see that the assumption (1) of Theorem \ref{thm11} is satisfied. 

Next we show that the assumption (2) of Theorem \ref{thm11} is satisfied as well, i.e. $\cd(\pi\times\pi)>\cd(\pi)$. 

We know that $\pi$ has a finite $K(\pi, 1)$ and hence there exists a finite free resolution $P_\ast$ of $\Z$ over $\zp$. Here each $\zp$-module $P_i$ is finitely generated and free and $P_i$ is nonzero only for finitely many $i$. Note that $P_\ast\otimes_\Z P_\ast$ is a free resolution of $\Z$ over $\zpp$ (see \cite{Brown}, Proposition 1.1 in chapter V). 
For any two left $\zp$-modules $A_1, A_2$ we have the natural isomorphism of chain complexes 
$$\Hom_\zp(P_\ast, A_1)\otimes_\Z \Hom_\zp(P_\ast, A_2) \to \Hom_\zpp(P_\ast\otimes_\Z P_\ast, A_1\otimes_\Z A_2).$$
If at least one of these chain complexes is flat over $\Z$, the K\"unneth theorem is applicable and we obtain the monomorphism
$$ H^j(\pi, A_1)\otimes_\Z H^{j'}(\pi, A_2)\to H^{j+j'}(\pi\times \pi, A\otimes_\Z A_2)$$
given by the cross-product $\alpha\otimes \alpha'\mapsto \alpha\times \alpha'$. 
We see that a certain cross-product $\alpha\times \alpha'$ is nonzero provided that the tensor product of the abelian groups $H^j(\pi, A_1)\otimes_\Z H^{j'}(\pi, A_2)$ is nonzero. 

Let $A_1, A_2$ be left $\zp$-modules chosen such that $H^n(\pi, A_1)\not=0$, where $n=\cd(\pi)$, and $H^1(\pi, A_2)\simeq \Z$. 
We may take $A_1=\zp$, see \cite{Brown}, chapter VIII, Proposition 6.7. 
Hence $\Hom_\zp(P_\ast, A_1)$ is a chain complex of free abelian groups. 
For some nonzero
elements $\alpha\in H^n(\pi, A_1)$ and $\alpha'\in H^1(\pi, A_2)$ we shall have $\alpha\times\alpha'\not=0$ since 
$H^n(\pi, A_1)\otimes_\Z H^{1}(\pi, A_2)\simeq H^{n}(\pi, A_1)\not=0.$ Thus, we obtain $\cd(\pi\times \pi)\ge n+1= \cd(\pi)+1$ as required.

We explain below how to construct the module $A_2$. 
Fix a non-unit element $g\in \pi$ and let $C=C_g\subset \pi$ denote the conjugacy class of $g$. Let $A_2=\Z^C=\Hom_\Z(\Z[C], \Z)$ denote the set of all functions $f:C\to \Z$; here $\Z[C]$ denotes the free abelian group generated by 
$C$.
Since $\pi$ acts on $C$ (via conjugation) we obtain the induced action of $\pi$ on $\Z^C$, where $(g\cdot f)(h)=f(g^{-1}hg).$ 

Let $P_\ast$ denote a free resolution of $\Z$ over $\zp$. We have the isomorphisms of chain complexes (which are similar to those used in the proof of Theorem \ref{thm5})
\begin{eqnarray*}\Hom_\zp(P_\ast, A_2)&\simeq& \Hom_\zp(P_\ast\otimes_\Z\Z[C], \Z)\\
&\simeq& \Hom_\zp(\Z[C], \Hom_\Z(P_\ast, \Z))\\ 
&\simeq& \Hom_{\Z[Z(g)]}(P_\ast, \Z).\end{eqnarray*}
The complex $P_\ast$ is a free resolution of $\Z$ over $\Z[Z(g)]$ by restriction. Thus, $H^i(\pi, A_2) \simeq H^i(Z(g), \Z)$ for any $i$. 
Since we know that the centraliser $Z(g)$ is isomorphic to $\Z$ we have 
$H^1(\pi, A_2)\simeq \Z$. 

This completes the proof. 
\end{proof}
%
%
%
%
%
\section{An application}\label{sym}

A symplectic manifolds $(M,\omega)$ is said to be {\it symplectically aspherical} \cite{KRT} if 
\begin{eqnarray}\label{van}\int_{S^2}\, f^\ast\omega \, =\, 0\end{eqnarray}
for any continuous map $f:S^2\to M$. 
It follows from the Stokes' theorem that a symplectic manifold which is aspherical in the usual sense is also symplectically aspherical.

\begin{theorem} \label{symp}
Let $(M, \omega)$ be a closed $2n$-dimensional symplectically aspherical manifold. Suppose that the fundamental group $\pi=\pi_1(M)$ is of type FL and 
the centraliser of any nontrivial element of $\pi$ is cyclic. Then the topological complexity $\tc(M)$ is either $4n$ or $4n+1$. 
\end{theorem}
\begin{proof} Our assumption on $\pi$ being of type $FL$ implies that there exists a finite cell complex $K=K(\pi, 1)$ (see \cite{Brown}, chapter VIII, Theorem 7.1).
Let $g: M\to K$ be a map inducing an isomorphism of the fundamental groups. Let $u\in H^2(M, \R)$ denote the class of the symplectic form. 
Then $u^{n}\not=0$ since $\omega^n$ is a volume form on $M$. 

Next we show that there exists a unique class $v\in H^2(K,\R)$ with $g^\ast(v)=u$. With this goal in mind we consider the Hopf exact sequence
(see \cite{Brown}, chapter II, Proposition 5.2)
$$\pi_2(M) \stackrel{h}\to H_2(M)\stackrel{g_\ast}\to H_2(K)\to 0$$
 where $h$ denotes the Hurewicz homomorphism and all homology groups are with integer coefficients. We may view $u$ as the homomorphism 
$u_\ast: H_2(M)\to \R$; it vanishes on the image of $h$ due to (\ref{van}). Hence there exists a unique $v_\ast: H_2(K)\to \R$ with 
$u_\ast= v_\ast\circ g_\ast$ and this $v_\ast$ is the desired cohomology class $v\in H^2(K, \R)$. 

Denote $\bar u = u\times 1-1\times u\in H^2(M\times M, \R)$ and also $\bar v = v\times 1-1\times v\in H^2(K\times K, \R)$. These classes are zero-divisors and $(g\times g)^\ast(\bar v) = \bar u$. Note that $$\bar u^{2n} =\pm \,  \binom {2n} n u^n\times u^n \not =0.$$ Thus we see that $\bar v^{2n} \not=0$ since $\bar u^{2n} = 
(g\times g)^\ast(\bar v^{2n})$. 

Applying Theorem \ref{thm2} to the class $\alpha = \bar v^{2n}\in H^{4n}(K, \R)$ we obtain $\wgt(\alpha)\ge 4n-1$. 

We claim that   
$\wgt(\bar u^{2n}) = \wgt((g\times g)^\ast(\alpha) ) 
\ge 4n-1$. Indeed, let $f: Y\to M\times M$ be a map satisfying the properties of the Definition \ref{defweight} with $k=4n-1$. 
Then 
$$f^\ast(\bar u^{2n}) =  f^\ast((g\times g)^\ast(\alpha) ) = \left[ (g\times g)\circ f   \right]^\ast (\alpha) =0.$$
Since $ \bar u^{2n}\not=0$, the inequality 
$\wgt(\bar u^{2n}) \ge 4n-1$ implies that $\tc(M)\ge 4n$. 
The upper bound $\tc(M)\le 4n+1$ is standard
(see \cite{F}, Theorem 4). Thus, $\tc(M) \in \{4n, 4n+1\}$ as claimed. This completes the proof.

\end{proof} 

\bibliographystyle{amsalpha}

\end{document}